%% file: Ramanujan.tex
\documentclass[11pt,a4paper]{amsart}
\pdfoutput=1

\pdfpageattr{/Group <</S /Transparency /I true /CS /DeviceRGB>>}

\usepackage[utf8]{inputenc}
\usepackage[T1]{fontenc}
\usepackage[american]{babel}
\usepackage[autostyle=true]{csquotes}
\usepackage{geometry}
\usepackage{amsmath, amsthm, amssymb}
\usepackage{mathtools}
\usepackage{xspace}
\usepackage{booktabs}
\usepackage{algorithm}
\usepackage{algorithmicx}
\usepackage{algpseudocode}
 
\usepackage{tikz}
\usetikzlibrary{arrows}
\usetikzlibrary{arrows.meta}
\usetikzlibrary{automata,petri}
\tikzset{%
  transition/.style={rectangle,minimum size=6mm,draw},
  place/.style={circle,minimum size=6mm,draw},
  database/.style={
    minimum width=2cm,minimum height=1cm,cylinder,
    shape border rotate=90,aspect=0.2,draw
  }
}

\usepackage{pgfplots}
\usepgfplotslibrary{external}
\usepackage{graphicx}
\usepackage[backend=biber]{biblatex}
\usepackage{enumitem}
\usepackage{epstopdf}
\usepackage{hyperref}

\theoremstyle{plain}
\newtheorem{theorem}{Theorem}
\newtheorem{lemma}[theorem]{Lemma}
\newtheorem{proposition}[theorem]{Proposition}
\newtheorem{corollary}[theorem]{Corollary}
\newtheorem{observation}[theorem]{Observation}

\newtheorem*{claim}{Claim}
\theoremstyle{definition}

\newtheorem{example}[theorem]{Example}

\newcommand{\RR}{\mathbb{R}}
\newcommand{\ZZ}{\mathbb{Z}}
\newcommand{\NN}{\mathbb{N}}
\DeclareMathOperator{\vol}{vol}
\DeclareMathOperator{\conv}{conv}
\DeclareMathOperator{\Bin}{Bin}
\newcommand\gfan{\texttt{Gfan}\xspace}
\newcommand\topcom{\texttt{TOPCOM}\xspace}
\newcommand\mptopcom{\texttt{MPTOPCOM}\xspace}

\newcommand\matlab{{\tt MATLAB} \xspace}

\newcommand{\pointconf}{P} 

\newcommand{\unvisited}{U}
\newcommand{\visited}{W}
\newcommand{\queue}{Q}
\newcommand{\processed}{P}

\newcommand\expectation{\mathbb{E}}

\DeclareMathOperator\LPS{LPS}
\DeclareMathOperator\PGL{PGL}
\DeclareMathOperator\PSL{PSL}
\DeclarePairedDelimiter\norm{\lVert}{\rVert}

\title{Random growth on a Ramanujan graph} 
\author{Janko B\"ohm}
\address[J. B\"ohm]{Technische Universität Kaiserslautern
Gottlieb-Daimler-Stra\ss e, Geb. 48, 67663 Kaiserslautern, Germany}
\thanks{%
Research by J.~B\"ohm on this work has been supported by Project II.5 of SFB-TRR 195:
\enquote{Symbolic Tools in Mathematics and their Application} of Deutsche Forschungsgemeinschaft.
}
\email{\texttt{boehm@mathematik.uni-kl.de}}
\author{Michael Joswig}
\thanks{%
Research by M.~Joswig is supported by Deutsche Forschungsgemeinschaft (EXC 2046: \enquote{MATH$^+$}, SFB-TRR 109: \enquote{Discretization in Geometry and Dynamics}, SFB-TRR 195: \enquote{Symbolic Tools in Mathematics and their Application}, and GRK 2434: \enquote{Facets of Complexity}).
}
\author{Lars Kastner}
\thanks{%
Research by L. Kastner is supported by Deutsche Forschungsgemeinschaft (SFB-TRR 195: ``Symbolic Tools in Mathematics and their Application'').
}
\author{Andrew Newman}
\thanks{Research by A. Newman is supported by Deutsche Forschungsgemeinschaft  Graduiertenkolleg ``Facets of Complexity" (GRK 2434)}
\address[M. Joswig, L. Kastner, and A. Newman]{
Institut f{\"u}r Mathematik, MA 6-2,
Technische Universität Berlin,
Str.\ des 17. Juni 136, 10623 Berlin, Germany
}
\email{\texttt{\{joswig, kastner, newman\}@math.tu-berlin.de}}


\begin{document}

\maketitle

\begin{abstract}
  The behavior of a certain random growth process is analyzed on arbitrary regular and non-regular graphs.
  Our argument is based on the Expander Mixing Lemma, which entails that the results are strongest for Ramanujan graphs, which asymptotically maximize the spectral gap.
  Further, we consider Erd\H{o}s--Rényi random graphs and compare our theoretical results with computational experiments on flip graphs of point configurations.
  The latter is relevant for enumerating triangulations.
\end{abstract}


\section{Introduction}

Many collections of mathematical objects can be equipped with a graph structure, and so their enumeration can be considered as visiting all nodes of such a graph.
Standard methods for visiting all nodes in a graph, say $G=(V,E)$, include the depth-first search and breadth-first search algorithms.
The common theme is that the set of visited nodes strictly grows which each step, until it covers the entire node set $V$.
This naturally leads to the following vast generalization.
We call any increasing sequence $P_1\subseteq P_2 \subseteq \cdots \subseteq V$ of node sets a \emph{growth process} on $G$, provided that $\bigcup_i P_i=V$.
Our main tool is a certain random growth process, Algorithm~\ref{algo:growth}, which maintains a \emph{queue} in addition to the randomly constructed set $P_t$ of nodes which have been processed at time~$t$.
This will be employed to estimate the number of nodes of a finite graph, which is given only implicitly.
Similar in spirit is a random sampling algorithm for estimating the size of a tree by Hall and Knuth~\cite{HallKnuth:1965,Knuth:1975}.

Naturally, estimating the size of a graph via a random process can only yield meaningful results if that graph is somehow \enquote{well connected}.
While there are several notions measuring this, here we we settle for expansion expressed in terms of spectral properties of the adjacency matrix; cf.\ \cite{AS,HLW}.
For a $d$-regular graph the largest eigenvalue is $d$, and the second largest (in absolute value), denoted $\lambda$, defines the \emph{spectral gap} $d-\lambda$.
Our first main result, Theorem~\ref{thm:main2}, is a lower bound on the expected size of the queue during a randomized breadth-first search through an arbitrary regular graph.
That lower bound becomes tighter the larger the spectral gap is, and the asymptotic optimum is marked by the regular Ramanujan graphs.
The proof of Theorem~\ref{thm:main2} is based on the Expander Mixing Lemma.
In Theorem~\ref{thm:VertexCounts} we use that result to also derive upper and lower bounds for the number of nodes of a regular graph from data obtained through randomized breadth-first search.
We experimentally check our estimates on a regular Ramanujan graph constructed by Lubotzky, Phillips and Sarnak~\cite{LubotzkyPhillipsSarnak:1988}.

Our motivation stems from enumerating triangulations of point sets in Euclidean space.
Standard methods are based on traversing the flip graph of the point configuration, which has a node for each triangulation, and the edges correspond to local modifications known as flips (or Pachner moves).
To estimate the number of triangulations from the input alone is next to impossible; see \cite[\S8.4]{Triangulations} for what is known.
We explore how our Theorems~\ref{thm:main2} and \ref{thm:VertexCounts} can help.
Again we do experiments, this time on the flip graphs of convex polygons, a topic intimately linked to the combinatorics of Catalan numbers.
These particular flip graphs are regular, but the exact spectral expansion of these graphs seems to be unknown.
An explicit computation reveals that the flip graph of a $k$-gon is a regular Ramanujan graph for $k\leq 7$, whereas the spectral gap seems to vanish pretty quickly for higher values of~$k$.

Most flip graphs are neither regular, nor do they appear to exhibit good expansion.
We address both issues separately.
First, we need to generalize our results to the non-regular setting.
We follow the standard approach to study the spectral expansion of non-regular graphs via the normalized adjacency matrix; cf.\ \cite{CG}.
Our main contribution here is a specific non-regular version of the Expander Mixing Lemma, which seems to be new.
This allows to interpret Theorems~\ref{thm:main2} and \ref{thm:VertexCounts} also in the non-regular setting.
Second, we study the expected asymptotic behavior of Algorithm~\ref{algo:growth} on random graphs in the Erd\H{o}s--Rényi model, both theoretically and experimentally.

The the paper is closed with yet one more experiment, on the quotient flip graph of the regular 4-cube.
The \enquote{quotient} comes from looking at triangulations modulo the natural symmetry of the cube.
The estimates resulting from Theorem~\ref{thm:VertexCounts} work surprisingly well, although the assumptions made in that result are not met by the quotient flip graph of the 4-cube.
Finally, again for the quotient flip graph of the 4-cube, we compare our estimates with those obtained from the Hall--Knuth sampling procedure \cite{HallKnuth:1965,Knuth:1975}.

\section{Expansion of regular graphs and a random growth process}\label{sec:RegularGraphs}
Here we describe a specific random growth process on arbitrary graphs which we ultimately want to study in the context of enumerating triangulations.
Yet we first set out to examine this random process in a much more idealized setting.
That is, we would like to understand its expected behavior on an \emph{expander graph}. 

Expander graphs have been studied quite extensively, and  we refer the reader to the survey of Hoory, Linial, and Wigderson \cite{HLW} or to \cite[Section 9.2]{AS} for an overview of many of their definitions and properties. Roughly speaking, for our purposes, an expander graph is a sparse graph on which a random walk mixes quickly.
To make this precise in our setting, we will introduce Ramanujan graphs. 

Let $G=(V,E)$ be an undirected graph with adjacency matrix $A$.
The \emph{eigenvalues} of $G$ are defined to be the eigenvalues $\lambda_1,\dots,\lambda_n$ of $A$.
Since $A$ is a symmetric matrix its eigenvalues are real, so we may assume that $\lambda_1 \geq \lambda_2 \geq \cdots \geq \lambda_n$.
Moreover, if $G$ is $d$-regular then $\lambda_1 = d$, and for all $i$ we have $|\lambda_i| \leq d$.
Let $\lambda(G) = \max\{|\lambda_2|, |\lambda_n|\} = \max\{|\lambda_i|\ |\ i=2,\ldots,n\}$.
A \emph{Ramanujan graph} $G$ is $d$-regular graph with $\lambda(G) \leq 2 \sqrt{d - 1}$.

The reader familiar with expander graphs will not find this topic of the spectrum of a graph surprising at all. On the other hand, to those unfamiliar with expander graphs this definition of Ramanujan graphs may seem strange in light of the rough description of expander graphs in terms of random walks. The well-known Expander Mixing Lemma, which likely first appeared in \cite{AlonChung} provides one (of many) connections between these two notions.
Moreover, the proof of our results will make extensive use of this lemma. 

The results we state here for $G$ a $d$-regular graph will be expressed in terms of $\lambda$, which we take as a shorthand for $\lambda(G)$.
Our results will be nontrivial whenever $\lambda<d$.
It is standard that this occurs exactly when $G$ is connected and not bipartite.
However, our results yield the strongest consequences when $d - \lambda$ is as large as possible.
That is in the case of Ramanujan graphs by a result of Alon \cite{Alon84,Alon91}, where the threshold $2\sqrt{d - 1}$ is explained as follows; see also Bilu and Linial \cite{BiluLinial:2006}.

\begin{theorem}[Alon--Boppana Theorem \cite{Alon84}]
For every $d \geq 2$ and every $\epsilon > 0$, there are only finitely many $d$-regular graphs $G$ with $\lambda \leq 2\sqrt{d - 1} - \epsilon$.
\end{theorem}

Before presenting our results, we give a precise statement of the Expander Mixing Lemma. We first define the notation $e(S, T)$ for a graph $G = (V, E)$ and $S,T \subseteq V$ to be the number of edges $(s,t)$ in $E$ where $s$ is in $S$ and $t$ is in $T$.
In particular, $e(S, T)$ counts edges with both endpoints in $S \cap T$ twice. 

\begin{theorem}[Expander Mixing Lemma]
  Let $G = (V, E)$ be a $d$-regular graph on $n$ vertices with $\lambda<d$.
  Then for any subsets $S$ and $T$ of $V$ we have
  \[ \left|e(S, T) - \frac{d\, |S|\, |T|}{n}\right| \ \leq \ \lambda \cdot \sqrt{|S|\, |T|\left(1 - \frac{|S|}{n}\right) \left(1 - \frac{|T|}{n}\right)} \enspace .\]
\end{theorem}

Here we are interested in the expected behavior of a randomized breadth-first search on a regular graph $G$. Toward stating and proving our results we recall the randomized breadth-first search algorithm and introduce the notation that we use.

\medskip

Our random growth process works on a connected graph $G$ as its input, and throughout the algorithm we keep track of three sets of vertices which always partition the vertices of~$G$.
The sets are denoted by $\processed$, $\queue$, and $\unvisited$ and respectively refer to the queued vertices, processed vertices, and unvisited vertices.

\begin{algorithm}
  \caption{Random Growth}
  \label{algo:growth}
  \begin{algorithmic}
    \Require{Some vertex $v_1$ of the connected graph $G$}
    \Ensure{Visit all vertices in $G$; i.e., the final state is  $\processed =V$, $\queue=\emptyset$, $\unvisited=\emptyset$}
    \State $\processed\leftarrow\emptyset$;\; $\queue\leftarrow\{v_1\}$;\;  $\unvisited\leftarrow V\setminus\{v_1\}$
    \While{$\queue\neq\emptyset$}
    \State pick a vertex $v$ uniformly at random from $\queue$
    \State $N\leftarrow\{w:w \text{ neighbor of } v \text{ in } U\}$
    \State  $\processed\leftarrow\processed\cup\{v\}$;\; $\queue\leftarrow(\queue\setminus\{v\})\cup N$;\; $\unvisited\leftarrow\unvisited\setminus N$
    \EndWhile
  \end{algorithmic}
\end{algorithm}

Occasionally we will use the variable $t \in \{0, 1, ..., n\}$ to refer to an arbitrary step of the process and when convenient, $\processed_t$, $\queue_t$, and $\unvisited_t$ refer to the sets $\processed$, $\queue$, and $\unvisited$,  at step $t$, with $|\processed_0| = 0$, $|\queue_0| = 1$ and $|\unvisited_0| = n - 1$.
Note that $|\processed_t| = t$ for all $t \in \{0, 1, ..., n\}$, and we denote by $v_t$ the unique vertex in $P_t\setminus P_{t-1}$.
The following is immediate.
\begin{lemma}\label{lemma:spanningtree}
  For each $t\geq 2$ let $p(t)$ be the minimal index in $\{1,2,\dots,t-1\}$ such that $v_t$ is a neighbor of $v_{p(t)}$.
  Then the edges $\{2,p(2)\},\{3,p(3)\},\dots,\{n,p(n)\}$ form a spanning tree of $G$ with root vertex $v_1$.
\end{lemma}

We seek bounds on the expected size of $\queue_t$ as $t$ varies from 0 to $n$.
As our results hold for large $n$, it makes sense to consider the densities of $\processed$, $\queue$, and $\unvisited$ rather than their sizes.
Let
\begin{equation} \label{eq:gamma-delta-epsilon}
   \pi:=|\processed|/n \,,\quad \kappa:=|\queue|/n \,, \quad \upsilon:=|\unvisited|/n \enspace.
\end{equation}
Thus, rather than fixing $n$ and deriving upper and lower bound for $\queue = \queue_t$ as $t$ varies from 0 to $n$, instead we derive upper and lower bounds for $\kappa \in [0, 1]$ in terms of $\pi \in [0, 1]$.

The first result is a \emph{structural lower bound} on $\kappa$ in terms of $\pi$ that derives from the Expander Mixing Lemma and the fact that there are never any edges from $\processed$ to $\unvisited$.
Note that this structural lower bound holds regardless of how we pick the next vertex in $\queue$. 

\begin{proposition}\label{prop:main}
  Let $G$, $d$, $n$, and $\lambda < d$ be as in the statement of the Expander Mixing Lemma.
  Further, let $t = \pi n$ be the number of steps completed by any random growth process on $G$.
  Then the density $\kappa$ of the queue $Q$ at step $t$ satisfies
  \[\kappa \ \geq \ 1 - \pi - \frac{\lambda^2(1-\pi)}{d^2\pi+\lambda^2(1-\pi)} \enspace .\]
\end{proposition}

\begin{proof} 
  Let $\processed = \processed_t$, $\queue = \queue_t$, and $\unvisited = \unvisited_t$, and their densities $\pi$, $\kappa$ and $\upsilon$ be as described above with $|\processed_t| = t = \pi n$.
  Due to $\pi + \kappa + \upsilon = 1$, an upper bound on $\upsilon$ implies a lower bound on $\pi +\kappa$.
  Thus we will upper bound $\upsilon$ in terms of $\pi$.
  
  Observe that $e(\processed, \unvisited) = 0$.
  Indeed, during the entire process a vertex is only moved from $\queue$ to $\processed$ once all of its neighbors have been found and added to $\queue$.
  On the other hand by the Expander Mixing Lemma we obtain
  \[ e(\processed, \unvisited) \ \geq \ \frac{d\, \pi\, n\, \upsilon \, n}{n} - \lambda \sqrt{ \pi \, n \, \upsilon \, n \, (1 - \pi) \, (1 - \upsilon)} \enspace. \]
  Thus we have
  \[0 \ \geq \ d \, \pi \, \upsilon - \lambda \sqrt{\pi(1 - \pi) \upsilon(1 - \upsilon)} \enspace.\]
  Since $\pi, \upsilon \in [0, 1]$, it follows that $\upsilon$ is bounded as
  \[\upsilon \ \leq \ \frac{\lambda^2(1-\pi)}{d^2\pi+\lambda^2(1-\pi)} \enspace.\]
  The claim follows since $\kappa = 1 - \pi - \upsilon$.
\end{proof}





While Proposition \ref{prop:main} is a deterministic statement that does not require any assumption on the order in which the vertices are processed, we obtain the following improved lower bound by introducing randomness.
\begin{theorem}\label{thm:main2}
  Let $G$, $d$, $n$, and $\lambda < d$ be as in the Expander Mixing Lemma.
  Further, let $t = \pi n$ be the number of steps completed by the random growth on $G$.
  Then the expected density $\kappa$ of the queue $Q$ at step $t$ satisfies
  \[\expectation(\kappa) \ \geq \ 1 - \pi - \exp \left(-\left(d - \lambda\right)\left(1 + \frac{1}{d - 1}\right) \pi \right) \enspace . \]
\end{theorem}
\begin{proof}

As in the proof of Proposition \ref{prop:main}, instead of $\kappa$ we will consider the density $\upsilon$ of unvisited vertices at time $t$. With $\processed_t$, $\queue_t$, $\unvisited_t$, $\pi$, $\kappa$, $\upsilon$ as before we derive upper and lower bounds on $\upsilon$ as a function of $\pi$. 
At any step $t$ we have by the Expander Mixing Lemma that the number of edges between $\queue_t \cup \processed_t$ and $\unvisited_t$ satisfies
\[ e(\processed_t \cup \queue_t, \unvisited_t) \ \geq \ \frac{d \, |\processed_t \cup \queue_t| \, |\unvisited_t|}{n} - \lambda \cdot \sqrt{|\processed_t \cup \queue_t| \, |\unvisited_t| \left(1 - \frac{|\processed_t \cup \queue_t|}{n}\right)\left(1 - \frac{|\unvisited_t|}{n}\right)} \enspace .\]

As there are no edges between $\processed_t$ and $\unvisited_t$, this also serves as a lower bound for $e(\queue_t, \unvisited_t)$.
At step $t$, the expected number of vertices added to the queue at the next step equals
\begin{equation}
  e(\queue_t, \unvisited_t)/|\queue_t| \enspace .
\end{equation}
This expectation has the following lower bound
\begin{eqnarray*}
\frac{e(\queue_t, \unvisited_t)}{|\queue_t|} &\geq&  \frac{d \, |\processed_t \cup \queue_t| \, |\unvisited_t|}{|\queue_t| \, n} - \frac{\lambda}{n\,|\queue_t|} \sqrt{|\processed_t \cup \queue_t| \, |\unvisited_t| \, (n - |\processed_t \cup \queue_t|) \, (n - |\unvisited_t|)}\\
&=& \frac{d - \lambda}{n} |\unvisited_t| \left(1 + \frac{|\processed_t|}{|\queue_t|}\right) \enspace .
\end{eqnarray*}
Now $|\queue_t|$ clearly satisfies the upper bound $|\queue_t| \leq (d - 1)t = (d - 1)|\processed_t|$, thus we have 
\[ \frac{e(\queue_t, \unvisited_t)}{|\queue_t|} \ \geq \ \frac{d - \lambda}{n} |\unvisited_t| \left(1 + \frac{1}{d - 1}\right) \enspace .\]
Therefore, the expected number of vertices removed from $\unvisited_t$ at step $t + 1$ given the size of $\unvisited_t$ is at least $\frac{d - \lambda}{n} |\unvisited_t| \left(1 + \frac{1}{d - 1}\right)$.
Thus, we may bound $\expectation(|\unvisited_{t + 1}|)$ from above as follows:
\begin{eqnarray*}
\expectation(|\unvisited_{t + 1}|) &=& \sum_{u = 0}^n \expectation(|\unvisited_{t + 1}| \mid |\unvisited_t| = u)\Pr(|\unvisited_t| = u)\\
&\leq& \left(1 - \frac{d - \lambda}{n}\left(1 + \frac{1}{d - 1}\right)\right) \sum_{u = 0}^n u \Pr(|\unvisited_t| = u) \\
&=& \left(1 - \frac{d - \lambda}{n}\left(1 + \frac{1}{d - 1}\right)\right) \expectation(|\unvisited_t|) \enspace .
\end{eqnarray*}
Now as $\unvisited_0$ is always the full graph minus a single vertex we have that 
\begin{eqnarray*}
\expectation(|\unvisited_t|) &\leq& \left(1 - \frac{d - \lambda}{n}\left(1 + \frac{1}{d - 1}\right)\right)^t (n-1) \\
&\leq& \exp \left(-\left(d - \lambda\right)\left(1 + \frac{1}{d - 1}\right) \frac{t}{n} \right)(n - 1) \enspace .
\end{eqnarray*}
Thus we get 
\[ \expectation(\kappa) \ \geq \ 1 - \pi - \exp \left(-\left(d - \lambda\right)\left(1 + \frac{1}{d - 1}\right) \pi \right) \enspace . \qedhere\]
\end{proof}

If one wants to use our random growth process to estimate the size of a graph, the Expander Mixing Lemma may in some sense be reversed.
We make this formal as follows.
As in the Expander Mixing Lemma and in Theorem \ref{thm:main2}, the strongest results follow when the spectral gap $d - \lambda$ is as large as possible, that is, for a Ramanujan graph. 
\begin{theorem}\label{thm:VertexCounts}
  Let $G$, $d$, $n$, and $\lambda < d$ be as in the Expander Mixing Lemma.
  For a given step $t$ in the random growth process, let $\processed$ denote the vertices that have been processed, $\queue$ denote the queue, $\visited$ denote the visited vertices (that is, $\visited = \processed \sqcup \queue$) and $\unvisited$ denote the unvisited vertices.
  Then $n$ satisfies
  \[ \frac{(d - \lambda) \cdot |\visited|^2}{(d - \lambda)\cdot|\visited| - e(\unvisited, \visited)} \ \geq \ n \ \geq \ \frac{(d + \lambda)\cdot|\visited|^2}{(d + \lambda)\cdot|\visited| - e(\unvisited, \visited)} \enspace .\]
\end{theorem}
\begin{proof}
  Recall from \eqref{eq:gamma-delta-epsilon} that $\upsilon$ is the density of $\unvisited$.
  The Expander Mixing Lemma directly implies the following upper bound and lower bounds on $\upsilon$:
  \[ \frac{(d + \lambda)\cdot|\unvisited|\cdot|\visited|}{n|\queue|} \ \geq \ \frac{e(\unvisited, \visited)}{|\queue|} \ \geq \ \frac{(d - \lambda)\cdot|\unvisited|\cdot|\visited|}{n|\queue|} \enspace. \]
  Therefore,
  \[ \frac{e(\unvisited, \visited)}{(d + \lambda)\cdot|\visited|} \ \leq \ \upsilon \ \leq \ \frac{e(\unvisited, \visited)}{(d - \lambda)\cdot|\visited|} \enspace .\]
  Now from the definition of $\upsilon$, we have $|\visited| = |\processed| + |\queue| = n - \upsilon n$.
  This yields $n = {|\visited|}/{(1 - \upsilon)}$, and the bounds in the statement follow from the bounds on $\upsilon$. 
\end{proof}

\begin{figure}[tb]
\centering
\includegraphics{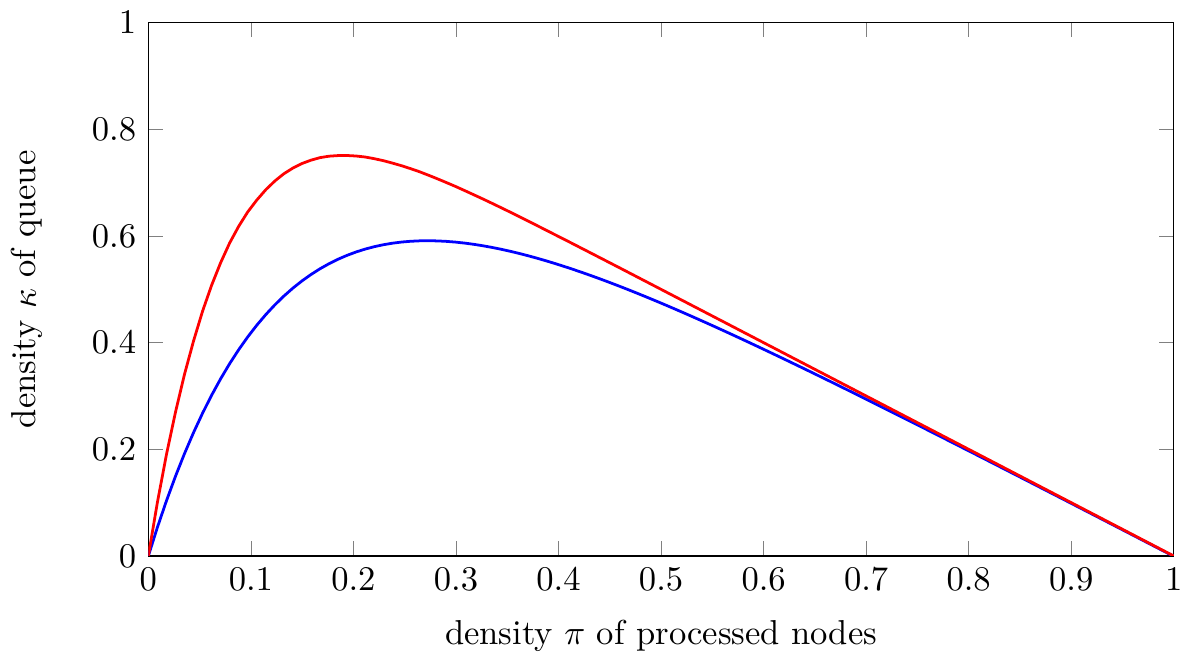}
\caption{Lower bound on the expected queue size from Theorem \ref{thm:main2} (blue) and a comparison to experimental data (red) for $\LPS(13,61)$.
}
\label{figure:RamanujanQueue}
\end{figure}

In our applications discussed in later sections often we will not know the value of $\lambda$ and our results described above will serve as heuristics to model what we see in our experiments.
However, before moving on to such applications, we describe a particular case where we do have a Ramanujan graph $G$. 

The celebrated 1986 paper of Lubotzky, Phillips, and Sarnak \cite[\S2]{LubotzkyPhillipsSarnak:1988} gave the first explicit construction of Ramanujan graphs.
For two distinct primes $p$ and $q$ both congruent to 1 mod 4, Lubotzky, Phillips, and Sarnak consider the Cayley graph of $p+1$ specially chosen generators of the projective linear groups $\PSL(2,\ZZ/q\ZZ)$ or $\PGL(2,\ZZ/q\ZZ)$, depending whether $p$ is a square modulo $q$ or not.
This yields a $(p + 1)$-regular Ramanujan graph, which we will denote $\LPS(p, q)$.
When $p$ is a square modulo $q$, the Ramanujan graph $\LPS(p, q)$ will be nonbipartite with exactly $(q^3 - q)/2$ vertices. 

For comparison with our results we ran Algorithm~\ref{algo:growth} on the Lubotzky--Phillips--Sarnak construction with $p = 13$ and $q = 61$.
Our reasoning for these values of $p$ and $q$ is that we want a Ramanujan graph which is of similar edge density to the flip graphs we consider later.
The resulting graph $\LPS(13,61)$ is 14-regular and has $113{,}460$ vertices. 
A \matlab \cite{MATLAB:2019a} computation shows that $\lambda \approx 7.1835$, which is smaller than $2\cdot \sqrt{13} \approx 7.2111$, and this confirms that $\LPS(13,61)$ is a Ramanujan graph.
In Figure~\ref{figure:RamanujanQueue} we show the comparison between the density of the queue throughout the search and the lower bound on the expected queue size described by Theorem \ref{thm:main2}.

\begin{figure}[tb]
\centering
\includegraphics{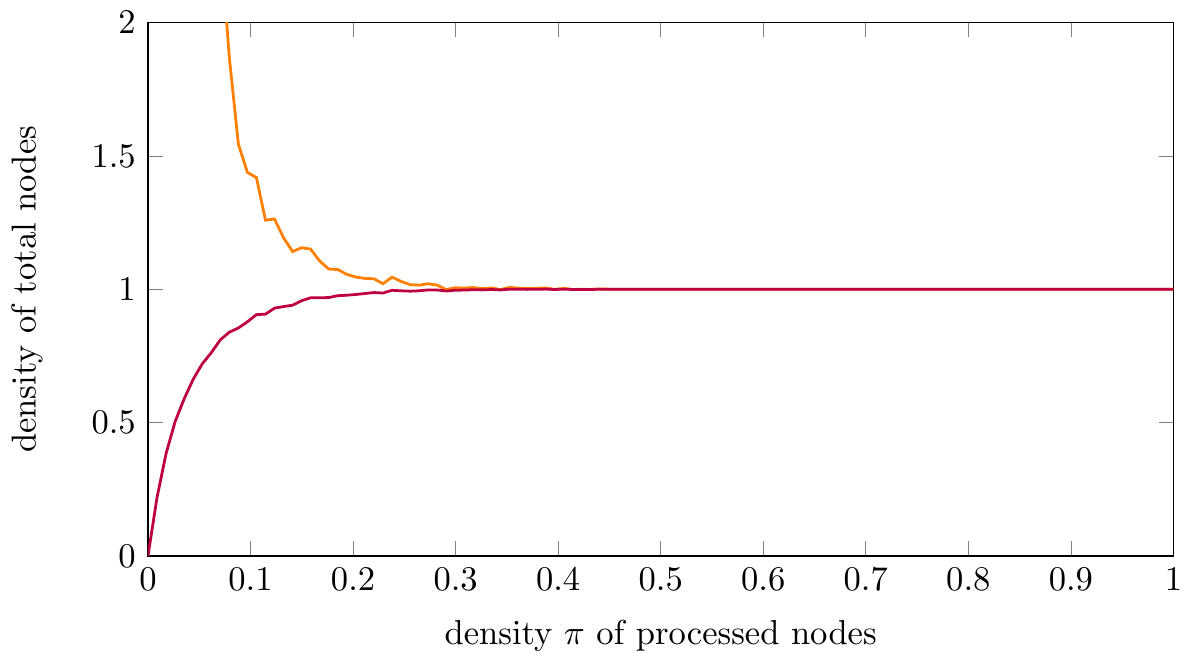}
\caption{
Upper and lower estimates on the total number of vertices (normalized by the actual number of vertices) of $\LPS(13,61)$.
}
\label{figure:RamanujanSize}
\end{figure}

Next, we consider how well Theorem \ref{thm:VertexCounts} does at predicting the size of the graph during the search process. For the experiment on $G = \LPS(13, 61)$, we start at a single node and run the random growth process. At every 1000 steps in the process we output the size of the queue and the number of processed nodes. Moreover, we want to estimate $e(\unvisited, \visited)$. This is done via a random sample. Observe that all edges from $\visited$ to $\unvisited$, have their $\visited$ endpoint contained $\queue$. Therefore to estimate $e(\unvisited, \visited)$, we sample 100 vertices from $\queue$ to estimate the average number of edges a vertex in $\queue$ sends to $\unvisited$.

For example, in the particular run we consider here, when there were 26{,}000 vertices processed, we had 84{,}102 vertices in the queue. We also know that the graph is 14-regular so at this point we have that \emph{with certainty} the number of vertices is at most
\[\frac{(14 - \lambda)(26{,}000 + 84{,}102)^2}{(14 - \lambda)(26{,}000+84{,}102) - e(\unvisited, \visited)} \enspace .\]

To estimate bounds however we need estimates on $\lambda$ and $e(\unvisited, \visited)$.
Since we have a Ramanujan graph $\lambda$ is at most $2\sqrt{13}$, and we take this as the estimate for~$\lambda$.
In this particular case a random sampling of size 100 from the vertices in the queue gave an estimate on $e(\unvisited, \visited)$ of $37{,}004.88$.
Thus the \emph{estimated} upper bound on the number of vertices is
\begin{equation}\label{eq:LPS-estimate-nodes}
  \frac{(14 - 2\sqrt{13})(26{,}000 + 84{,}102)^2}{(14 - 2\sqrt{13})(26{,}000 + 84{,}102) - 37{,}004.88} \ \approx \ 115{,}836.7 \enspace.
\end{equation}
Similarly, one estimates a lower bound on the number of vertices of $111{,}874.7$.
So with only about 23 percent of the vertices moved into the set $P$, we can give upper and lower bounds on the total number of vertices that are both within 3 percent of the right answer.

Figure \ref{figure:RamanujanSize} shows the curves giving the upper and lower estimates throughout the process.
Note that Theorem \ref{thm:VertexCounts} gives no finite upper bound early in the process, and the figure reflects this.

Of course, as we mentioned above, in this particular case we do know a precise approximation for $\lambda$.
Replacing $2\sqrt{13}$ by $7.1835$ in \eqref{eq:LPS-estimate-nodes} improves the estimated upper bound for the number of nodes of $\LPS(13,61)$ only slightly, to $115{,}812.27$.
However, in larger examples we would not be able to expect to compute $\lambda$ numerically, nor will that be the goal, so we want to instead focus on the coarser estimate here.

\section{Flip graphs of point configurations}\label{section:flip_graphs}

Let $\pointconf\subset\RR^\ell$ a finite set of $n$ points that affinely spans the entire space.
A \emph{triangulation} $\Sigma$ of $\pointconf$ is a simplicial complex which covers the convex hull $\conv \pointconf$, such that the vertices of each simplex form a subset of the given points $\pointconf$; cf.\ \cite[\S2.3.1]{Triangulations}.
The set of all subdivisions of $\pointconf$ is partially ordered by refinement, and the triangulations are precisely the finest subdivisions.
The triangulations form the nodes of a graph, where the edges arise from local modifications known as \emph{flips}; cf.\ \cite[Definition 2.4.7]{Triangulations} and Figure~\ref{fig:ngon}.
This is the \emph{flip graph} of $\pointconf$.

A certain class of triangulations is of particular interest, e.g., due to connections with algebra \cite[\S1.3]{Triangulations}.
The triangulation $\Sigma$ is \emph{regular} if it is induced by a height function $h:\pointconf\to\RR$ in the sense that the lower convex hull of
\[
\conv\bigl\{ (p,h(p)) \mid  p\in \pointconf \bigr\} \quad \subset \ \RR^{\ell+1}
\]
projects to $\Sigma$ by omitting the last coordinate.
The subgraph of the flip graph whose nodes correspond to the regular triangulations is the \emph{flip graph of regular triangulations} of $\pointconf$.
In the literature this is often called the \enquote{regular flip graph}; however, as we also talk about regularity in the graph-theoretic sense we make effort to avoid ambiguity between the two notions of regular. 
Structurally, it is essential that the flip graph of regular triangulations is contained in the vertex-edge graph of a convex polytope, the \emph{secondary polytope} of $\pointconf$, which is defined up to normal equivalence \cite[Theorem 5.3.1]{Triangulations}.
In particular, the flip graph of regular triangulations is necessarily connected \cite[Corollary 5.3.14]{Triangulations}; in general, this does not hold for the flip graph of all triangulations \cite[\S7.3]{Triangulations}.

For simplicity of the exposition we will now assume that the points in $\pointconf$ are in convex position, i.e., they form the vertices of their convex hull, $\conv\pointconf$.
In this way we can also afford some sloppiness by not distinguishing between a polytope and its set of vertices.

\begin{figure}[bh]
  \includegraphics[width=.9\textwidth]{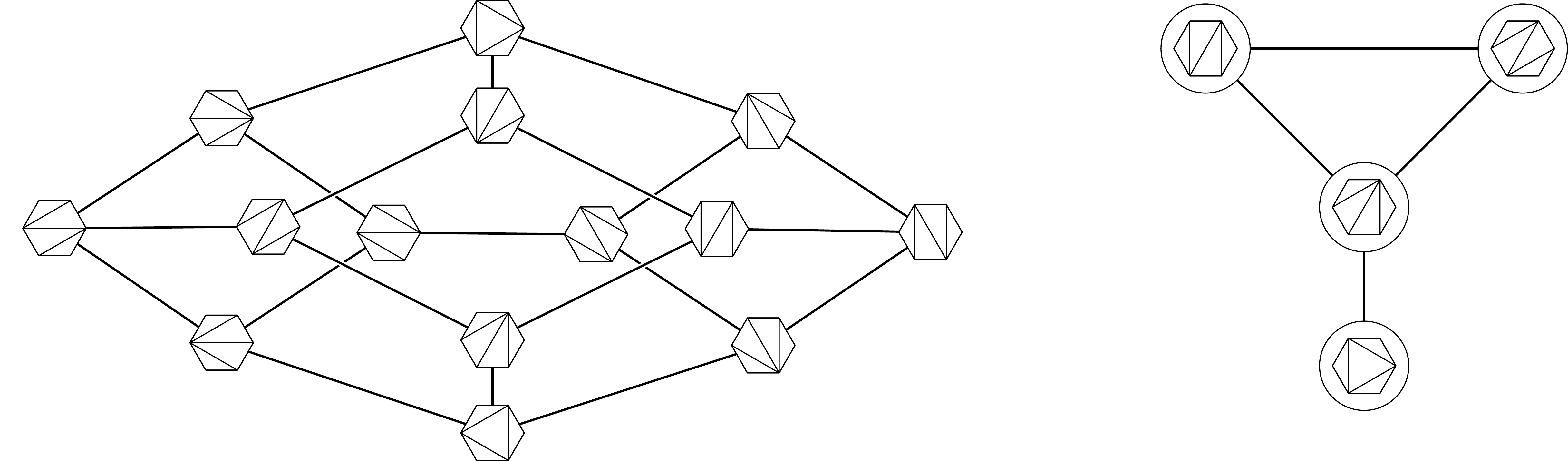}
  \caption{On the left: The flip-graph of a hexagon is the vertex-edge graph of the $3$-dimensional associahedron. On the right: The quotient flip graph for the cyclic group $C_6$ of rotations.}
  \label{fig:ngon}
\end{figure}


\begin{example}\label{exmp:ngon}
  Let $\pointconf\subset\RR^2$ be the set of vertices of a convex $k$-gon.
  This point configuration has $\tfrac{1}{k-1}\tbinom{2k-4}{k-2}$ triangulations, which is a Catalan number.
  Each triangulation is regular, and it is determined by its $k-2$ diagonals, and each one gives rise to a flip.
  This entails that the flip graph of $\pointconf$ is a regular graph of degree $k-2$ with $\tfrac{1}{k-1}\tbinom{2k-4}{k-2}$ nodes.
  The $(k{-}3)$-dimensional associahedron, which is simple, is a secondary polytope.
\end{example}

\input{ngon_lambdas}

The spectral expansion of the flip graph of a polygon is unclear, although the weaker condition of low diameter is known; cf.\ \cite{SleatorTarjanThurston:1988} and \cite{Pournin:2014}.
Yet inspecting the values for $\lambda$ and $2\sqrt{d-1}$ in Table~\ref{tab:flip-k-gon} yields the following.
\begin{observation}
  The flip graph of a $k$-gon is a regular Ramanujan graph for $k\in\{5,6,7\}$.
\end{observation}

In Figure \ref{figure:15gon}, we show the results of applying Theorem \ref{thm:VertexCounts} on the random growth process applied to the flip graph $\Phi$ of the 15-gon.
We estimate $e(\unvisited, \visited)/|\queue|$ at every 1000 steps by a random sampling of 100 vertices from $\queue$. More importantly, rather than taking the actual value of $\lambda(\Phi)$, we wish to see what Theorem \ref{thm:VertexCounts} would tell us were $\Phi$ to have good expansion so we take $\lambda= 2\cdot\sqrt{11} \approx 6.633$ to use Theorem \ref{thm:VertexCounts} to give upper and lower estimates on the number of vertices of $\Phi$. This is compared with the true number of vertices, 742,900.
In Figure \ref{figure:15gon} the upper curve shows the ratio of the upper estimate of Theorem \ref{thm:VertexCounts} to the true number of vertices and the lower curve shows the same ratio for the lower estimate.
\begin{figure}[htb]
\centering
\includegraphics{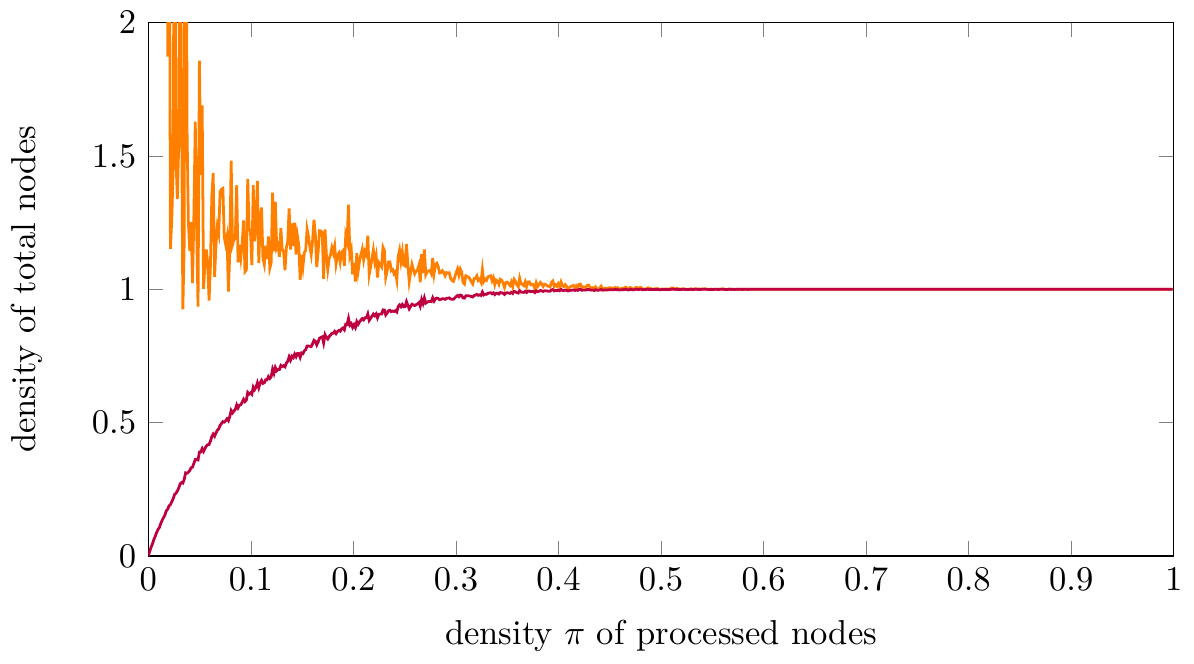}
\caption{
Upper and lower estimates on the total number of vertices of the flip graph of the 15-gon (normalized by the actual number of vertices)
}
\label{figure:15gon}
\end{figure}

We compare this figure back to the idealized setting of Figure \ref{figure:RamanujanSize}.
Here there is more volatility in the upper bound at the beginning and the two curves do not converge to each other quite as quickly. Still, the two curves become quite close to one another long before the process has finished.
Note that the flip graph of the 15-gon is 12-regular compared to the 14-regular $LPS(13, 61)$. 

Even when only 220,000 vertices have been processed in this particular run, there are 476,113 vertices waiting in the queue, the data predicts that the total number of vertices will be between 715,279 and 767,518. So with only 30 percent of the vertices processed and another 64 percent in the queue we already have the right number of vertices bounded between 96 percent and 103 percent of its true value.

Next we describe other polytopes whose flip graphs are regular graphs.
A \emph{split} of $\pointconf$ is a subdivision with exactly two maximal cells; it is necessarily regular \cite[Lemma 3.5]{HerrmannJoswig:2008}.
Splits are coarsest subdivisions, i.e., they yield minimal elements of the refinement poset, and thus they correspond to facets of any secondary polytope.
If, conversely, each coarsest subdivision is a split, then $P$ is called \emph{totally splittable}.
This is the case, e.g., when $\pointconf$ is the vertex set of a polygon: each diagonal defines a split; cf.\ Example~\ref{exmp:ngon} and Figure~\ref{fig:ngon}.
Note that every subdivision of a totally split point configuration is regular.
There is a full characterization of the totally split polytopes:

\begin{theorem}[{\cite[Theorem~9]{HerrmannJoswig:2010}}]\label{thm:totally-splittable}
  A polytope $\pointconf$ is totally splittable if and only if it has the same oriented matroid as a simplex, a crosspolytope, a polygon, a prism over a simplex, or a (possibly multiple) join of these polytopes.
\end{theorem}

This result has an immediate consequence on the associated flip graphs.

\begin{corollary}
  Let $\pointconf$ be the set of $k$ vertices of a totally splittable $\ell$-polytope.
  Then any secondary polytope is simple of dimension $k-\ell-1$.
  Consequently, the flip graph of regular triangulations of $\pointconf$ is $(k{-}\ell{-}1)$-regular.
\end{corollary}
\begin{proof}
  The dimension of the secondary fan modulo its lineality space equals $k-\ell-1$; cf.\ \cite[\S5.1.3]{Triangulations}.
  Now the claim follows from an inspection case by case.
  By \cite[Remark 11]{HerrmannJoswig:2010} the secondary polytopes of totally splittable polytopes are (possibly multiple) products of simplices, permutohedra, and associahedra.
\end{proof}


\section{Non-regular graphs}\label{sec:nonregular}
Most flip graphs are not regular.
Nonetheless, we aim for an analogue of Theorem~\ref{thm:main2} in the absence of regularity. 
In the most classic setting expander graphs are defined in terms of regular graphs. However in the literature there is a more general notion of spectral expansion for non-regular graphs in terms of the normalized adjacency matrix.
For our purposes we will keep the assumption that $G$ is a graph with no isolated vertices and let $D = D(G)$ denote the degree matrix of $G$, which is diagonal and invertible.
Then the normalized adjacency matrix of $G$ is defined as
\[N \ := \ D^{-1/2}AD^{-1/2} \enspace,\]
where $A$ is the usual adjacency matrix of $G$.
It is standard knowledge that the eigenvalues of $N$ fall in $[-1, 1]$, with the multiplicity of the eigenvalue 1 corresponding to the number of connected components of $G$; see, e.g., \cite{CG}.
We use $\mu_i$ to denote the eigenvalues of $N$ with $\mu_2 > \mu_3 > \cdots > \mu_n$, and we write $\mu := \mu(G)$ for the maximum of $\mu_2$ and $-\mu_n$. 

Before we justify comparing the flip graphs from our experiments to (non-regular) expander graphs, we present a non-regular version of the Expander Mixing Lemma. This will allow us to be more precise when analyzing more experiments in Section \ref{section:in_practice}. 
We first need to define the volume of a set of vertices in a graph. For a graph $G$ and a subset $U$ of the vertices of $G$, the \emph{volume} of $U$ is
\[\vol(U) \ := \ \sum_{v \in U} \deg(v) \enspace .\]
There are several versions of the Expander Mixing Lemma for non-regular graphs known that use this volume notion; see, e.g., \cite{CG}. 
However, none that we found gave precisely the formulation that we need for our purposes, so we give our own formulation here.
Our argument essentially follows the proof of the $d$-regular analogue found in \cite[Lemma 2.1]{AlonChung}.

\begin{proposition}\label{prop:generalExpMix}
  Let $G = (V, E)$ be a graph with $\mu = \mu(G)$.
  Then for any partition of $V$ into two nonempty set $S$ and $T$, one has
  \[\left| e(S, T) - \frac{\vol(S)\vol(T)}{\vol(V)} \right| \ \leq \ \mu \frac{\vol(S) \vol(T)}{\vol(V)} \enspace .\]
\end{proposition}

\begin{proof}
Recall that $N = D^{-1/2}AD^{-1/2}$. Then $f: V(G) \rightarrow \RR$ defined by $f(v) = \sqrt{\deg(v)}$ is an eigenvector for the eigenvalue 1. For $U$ and $V$ given as in the statement let $g : V(G) \rightarrow \RR$ be defined by $g(v) = -\sqrt{\deg(v)}/\vol(S)$ if $v \in S$ and $g(v) = \sqrt{\deg(v)}/\vol(T)$ if $v \notin S$. Now $f$ is orthogonal to $g$ since 
\[
  \langle f, g \rangle \ = \ \sum_{v \in S} -\frac{\deg(v)}{\vol(S)} + \sum_{v \in T} \frac{\deg(v)}{\vol(T)} 
    \ = \ -\frac{\vol(S)}{\vol(S)} + \frac{\vol(T)}{\vol(T)} \ = \ 0 \enspace .
\]
It follows from the fact that $N$ is a symmetric matrix (and hence has an orthogonal eigenbasis) therefore that 
\[
  |\langle Ng, g \rangle| \ \leq \ \mu \norm{g}^2 \enspace .
\]
Also, $\norm{g}^2 = \tfrac{1}{\vol(S)} + \tfrac{1}{\vol(T)}$.
It remains to compute $\langle Ng, g \rangle$.
We have
\[
  \langle Ng, g \rangle \ = \ \langle D^{-1/2} A D^{-1/2} g, g \rangle \ = \ \langle AD^{-1/2} g, D^{-1/2} g \rangle \enspace .
\]
Let $g' = D^{-1/2} g$.
Thus $g'(v) = -\tfrac{1}{\vol(S)}$ if $v \in S$ and $g'(v) = \tfrac{1}{\vol(T)}$ if $v \in T$.
It is easy to check that
\[
  \langle A g', g' \rangle \ = \ 2 \sum_{ij \in E} g'(i)g'(j) \enspace;
\]
indeed this holds for any vector in $\RR^{|V|}$.
Now as an edge can either have both endpoints in $S$, both endpoints in $T$, or contribute to $e(S, T)$ we have
\[
  \begin{split}
    2 \sum_{ij \in E} g'(i) g'(j) \ &= \ \frac{e(S, S)}{\vol^2(S)} + \frac{e(T, T)}{\vol^2(T)} - \frac{2 e(S, T)}{\vol(S)\vol(T)} \\
    &= \ \frac{\vol(S) - e(S, T)}{\vol^2(S)} + \frac{\vol(T) - e(S, T)}{\vol^2(T)} - \frac{2e(S, T)}{\vol(S)\vol(T)} \\
    &= \ \frac{1}{\vol(S)} + \frac{1}{\vol(T)} - e(S, T) \left(\frac{1}{\vol^2(S)} + \frac{2}{\vol(S)\vol(T)} + \frac{1}{\vol^2(T)}\right) \\
    &= \ \frac{1}{\vol(S)} + \frac{1}{\vol(T)} - e(S, T) \left( \frac{1}{\vol(S)} + \frac{1}{\vol(T)} \right)^2 \\
    &= \ \norm{g}^2 \left( 1- e(S, T) \left(\frac{1}{\vol(S)} + \frac{1}{\vol(T)} \right) \right) \enspace .
  \end{split}
\]
Thus,
\[
  \left| 1- e(S, T) \left(\frac{1}{\vol(S)} + \frac{1}{\vol(T)} \right) \right| \ = \ \frac{\langle Ng,g\rangle}{\norm{g}^2} \ \leq \ \mu \enspace.
\]
Since $\dfrac{1}{\vol(S)} + \dfrac{1}{\vol(T)} = \dfrac{\vol(V)}{\vol(S) \vol(T)}$, the claim follows.
\end{proof}


As in the 15-gon example, we want to compare non-regular flip graphs to Ramanujan graphs. However, without regularity we need to specify what we mean by the term Ramanujan graph. There does already exist a notion of non-regular Ramanujan graphs, posed by Lubotzky in \cite{Lubotzky95} and discussed in \cite{BLM} stated in terms of the spectral radius of the universal covering tree of a graph. Here, we will not be so precise. Rather our assumption will be that the degree sequence of the graphs under consideration have small variance. This will allow us to draw a comparison to a regular Ramanujan graph.

In terms of the variance of the degree sequence we can obtain the following hybrid version of the Expander Mixing Lemma as a corollary to Proposition \ref{prop:generalExpMix}.
\begin{corollary}\label{cor:hybrid}
Let $G = (V, E)$ be a graph with $\mu = \mu(G)$ and so that $G$ has average degree $\overline{d}$, and the variance of its degree sequence is $\sigma^2$, then for any partition of $V$ into two sets $S$ and $T$ one has the following lower bound on $e(S, T)$:
\[
  e(S, T) \ \geq \ (1 - \mu) \frac{(\overline{d}|S| - \sigma \sqrt{|S|n})(\overline{d}|T| - \sigma \sqrt{|T| n})}{\overline{d}n} \enspace .
\]
\end{corollary}
\begin{proof}
  The proof follows from Proposition \ref{prop:generalExpMix} and the subsequent claim.
\end{proof}
\begin{claim}
For $G$ as in the statement and $S$ any subset of the vertex set of $G$ one has $|\overline{d}|S| - \vol(S)| \leq \sigma \sqrt{|S|n}$. 
\end{claim}
\begin{proof}[Proof of claim]
Let $G$ and $S$ be given as above. We have,
\begin{eqnarray*}
\left(\overline{d}|S| - \vol(S)\right)^2 &=& \left(\sum_{v \in S} (\overline{d} - \deg(v)) \right)^2 \\
&=& \left( \sum_{v \in G} (\overline{d} - \deg(v)) \textbf{1}_S(v) \right)^2,
\end{eqnarray*}
where $\textbf{1}_S$ denotes the indicator function of $S$. Thus by the Cauchy--Schwarz inequality the above quantity is at most
\begin{eqnarray*}
\sum_{v \in G} (\overline{d} - \deg(v))^2 \sum_{v \in G} (\textbf{1}_S(v))^2 &=& \sigma^2 n |S| \enspace.
\end{eqnarray*}
And the claim follows. 
\end{proof}
It follows from the statement of Corollary \ref{cor:hybrid} that for $\overline{d}$ much larger than $\sigma^2$ and $|S|$ and $|T|$ both far enough away from 0 we essentially have the regular Expander Mixing Lemma with the degree of regularity replaced by the average degree of the graph. The requirement that $|S|$ and $|T|$ both be far away from zero (where ``far away" is determined by how close $\sigma$ is to zero) effectively creates two blind spots in a predicted lower bound. We don't make a precise statement in terms of $\sigma^2$ here, but rather point out that Corollary \ref{cor:hybrid} together with the proof of Theorem \ref{thm:main2} gives us a heuristic explanation for applying Theorem \ref{thm:main2} with non-integer values of $d$. To be more precise we denote by $\beta(\pi, d, \lambda)$ the expression in the lower bound of Theorem \ref{thm:main2}, 
\[ \beta(\pi, d, \lambda) \ = \ 1 - \pi - \exp \left( -\left(d - \lambda\right)\left(1 + \frac{1}{d - 1}\right) \pi\right) \enspace. \]
In the next section, we make an experimental comparison between the behavior of the queue of a non-regular flip graph with average degree $\overline{d}$ and the behavior of the queue size modeled by $\beta(\pi, \overline{d}, 2 \sqrt{\overline{d} - 1})$

Recall that, by the Alon--Boppana Theorem, the threshold $2 \sqrt{d - 1}$ is essentially the smallest possible value for $\lambda(G)$ when $G$ is a $d$-regular graph.
Similarly, it follows from a result of Mohar \cite{Mohar} that $\lambda(G) = 2\sqrt{\overline{d} - 1}$ is the best one can hope for as far as the spectral gap of a non-regular graph is concerned.



\section{Random Graphs}
To complete our theoretical analysis of the random growth process in Algorithm~\ref{algo:growth} we now study its behavior on random graphs in the classical Erd\H{o}s--R\'{e}nyi model \cite{ErdosReyni:1960}.
This is particularly interesting in view of a result of Friedman \cite{Friedman:1991} which relates Ramanujan graphs with random graphs and because of which Ramanujan graphs are sometimes called \enquote{quasirandom}.
Note, however, that Friedman's random graphs are uniformly sampled from the class of all regular graphs.
The Erd\H{o}s--R\'{e}nyi random graphs, which we denote $G(n,p)$, form an easier model to sample from and are more commonly studied.
The random graph $G(n,p)$ has $n$ vertices and each of the $\tbinom{n}{2}$ edges appearing independently with probability~$p$.

In the previous section we looked at non-regular graphs of average degree $d$.
This suggests to compare with $G \sim G(n, d/n)$ as the expected degree of any fixed vertex under this distribution is $\tfrac{d}{n}(n - 1)$.
Before we will do this comparison a few words about the asymptotic behavior of $G \sim G(n, d/n)$ are in order.
More background on Erd\H{o}s--R\'{e}nyi random graphs may be found in, e.g., \cite{AS},  \cite{BollobasRandomGraphs}, or \cite{JansonLuczakRucinski}.
The first thing to mention is that $G \sim G(n, d/n)$ is asymptotically almost surely \emph{disconnected} if $n$ is large.

\begin{theorem}[Erd\H{o}s--R\'{e}nyi \cite{ErdosReyni:1960}]\label{ERPhaseTransition}
  Let $d$ be a fixed constant.
  If $d < 1$, then asymptotically almost surely $G \sim G(n, d/n)$ is a disjoint union of components of order $O(\log n)$.
  However, if $d > 1$, then asymptotically almost surely $G \sim G(n, d/n)$ has a unique \emph{giant component} on $(1 + o(1)) \delta_0 n$ vertices, where $\delta_0$ is the unique root of $1 - x = e^{-dx}$ in the open interval $(0, 1)$.
  In the latter case the remaining components have order $O(\log n)$.
\end{theorem}

This significant change in behavior at $d = 1$ is called the \emph{Erd\H{o}s--R\'{e}nyi phase transition} and has been extensively studied; details may be found in \cite[Chapter~9]{AS}.

Here we briefly want to look into the connection with Ramanujan graphs.
The Expander Mixing Lemma tells us that for $G$ a $d$-regular graph and any subsets $S$ and $T$ of the vertices, $e(S, T)$ is close to $\tfrac{d}{n}|S||T|$.
The term $\tfrac{d}{n}|S||T|$ is the expected number of edges between $S$ and $T$ if $G$ is taken to be an Erd\H{o}s--R\'{e}nyi graph distributed as $G(n, d/n)$.
Thus a Ramanujan graph has roughly the number of edges between $S$ and $T$ for any $S$ and $T$ as would be expected in a random graph of the same edge density, where \enquote{roughly} is measured by the spectral gap.
While Theorem~\ref{thm:main2} holds for a deterministic graph, here we add a level of randomness in running our random growth process on a random graph. 

Now we investigate our random growth process on an Erd\H{o}s--R\'{e}nyi random graph.
We should be a bit careful now that we have two levels of randomness, namely the graph and the random selection of vertices from the queue.
Note that we will only be interested in the case where the giant component exists.
We consider the following experiment for any $n \in \NN$ and $d > 1$:
\begin{algorithm}[H]
  \caption{Random Growth on a Random Graph}
  \label{algo:erdos-renyi}
  \begin{algorithmic}[1]
    \State Generate a graph $G$ uniformly at random from $G(n, d/n)$.
    \State Pick a vertex $v_1$ uniformly at random from $G$.
    \State Run Algorithm~\ref{algo:growth} starting at $v_1$, and return the tree described by Lemma \ref{lemma:spanningtree}.
  \end{algorithmic}
\end{algorithm}

By Theorem \ref{ERPhaseTransition}, we know that, as $n$ tends to infinity, the density of the final tree will either be $\delta_0$ with probability $\delta_0$, or it will be zero with probability $1 - \delta_0$, because with probability $1 - \delta_0$ we will pick a component on $O(\log n)$ vertices.
As we are interested in searching the giant component, we consider Algorithm~\ref{algo:erdos-renyi} conditioned on $v_1$ being a vertex in the giant component.
Moreover, to simplify the exposition, we let this process run to step $n$, with nothing happening after the queue first becomes empty.
In this way we do not have to concern ourselves with conditioning on the exact size of the giant component.
Even if we let Algorithm~\ref{algo:erdos-renyi} run for $n$ steps, we keep track of the first time $t$, where $Q_t = \emptyset$.
Then $P_t$ is the component of $v_1$ and $U_t$ comprises the vertices in the complement.
That is, $t=|P_t|$, and we define $P_r = P_t$, $Q_r = Q_t$, and $U_r = U_t$ for all $t < r \leq n$.
With this notation we are prepared to prove the following theorem about the growth of the queue in a random search of the giant component of a random graph.

\begin{theorem} \label{randomgraph}
  Fix $d > 1$ and let $H \subseteq G \sim G(n, d/n)$ be a largest component in the random graph in the Erd\H{o}s--R\'{e}nyi model.
  We run Algorithm~\ref{algo:growth} on $H$ starting at a vertex $v_1$ in $H$, with the modification to let it run for $n$ steps.
  Then for any $\pi \in (0, \delta_0)$ the asymptotic expected density of $Q_t$, where $t = \pi n$, equals $1 - \pi  - \exp(-d \pi)$.
  If, however, $\pi \in (\delta_0,1]$ the asymptotic expected density of $Q_t$ equals $0$.
\end{theorem}

Before we enter the proof, let us explain how the analysis of Algorithm~\ref{algo:growth} from Theorem~\ref{randomgraph} applies to Algorithm~\ref{algo:erdos-renyi}.
Since initially all vertices look alike, we do not know if $v_1$ is actually contained in the giant component.
However, this is easy to fix by keeping track of the sizes of the sets $P$, $Q$ and $U$ during in the random growth process.
This leads to a numeric process by defining $u_t$ and $q_t$ for $t \in \NN$, as
\begin{equation}\label{eq:numeric}
  \begin{aligned}
    u_t &= u_{t - 1} - X \,, \text{ where } X \sim \Bin(u_{t - 1}, d/n) \,, \quad u_0 = n \\ q_t &= n - t - u_t \enspace.
  \end{aligned}
\end{equation}
Of course, $q_t$ will with probability 1 eventually become negative.
The important observation is that from the start to the first time $t \in (0, n)$ where $q_t = 0$, these sequences measure the size of the sets in an instance of Algorithm~\ref{algo:erdos-renyi}.
We begin with the following simple lemma about this numeric process.

\begin{lemma}\label{lem:numeric}
  In the numeric process \eqref{eq:numeric} the asymptotic expected value of $u_t/n$ equals $\exp(-d \upsilon)$ for $t = \pi n$, where $\pi \in (0, 1]$ is fixed, and $n$ tends to infinity.
\end{lemma}

\begin{proof}
  From the definition of $\{u_t\}_{t \geq 1}$, we have that $u_t$ is distributed as a binomial random variable with $n$ trials and success probability $\left(1 - \frac{d}{n}\right)^t$.
  Indeed, we may sample $u_t$ by the following experiment: Begin with a set of $n$ vertices, and at each step in $1,\ldots,t$ choose a random set by including each vertex with independently probability $d/n$.
  This distribution induces the binomial distribution $\Bin(u_{t - 1}, d/n)$ at step $t$ on the vertices that have not yet been selected.
  Thus $u_t$ is given by the number of vertices that are never picked in this process.
  From this we have that $\expectation(u_t) = (1 - d/n)^{\upsilon n} n$.
  Thus the expected value of $u_t/n$ is asymptotically $\exp(-d \upsilon)$.
\end{proof}

Now we will prove Theorem~\ref{randomgraph}, which is about Algorithm~\ref{algo:growth}, by analyzing Algorithm~\ref{algo:erdos-renyi} and using Lemma~\ref{lem:numeric}.

\begin{proof}[Proof of Theorem~\ref{randomgraph}]
  From Lemma~\ref{lem:numeric} it follows that the expected value of $q_t/n$ at $t = \pi n$ is $1 - \pi - \exp(-d \pi)$.
  Of course, this only says something about Algorithm~\ref{algo:erdos-renyi} in the case that $q_r \geq 0$ for all $r \leq t$.
  The key observation is that after Steps 1 and~2 the component $H$ of the chosen vertex $v_1$ is deterministic, and the graph $H$ is connected.
  Thus for $\pi \in (0, 1)$, we have $|Q_r| > 0$ for all $r \leq \pi n$ if and only if $v_1$ lies in a component of $G \sim G(n, p)$ of order at least $\pi n$.
  By Theorem \ref{ERPhaseTransition}, the probability of this event, asymptotically in $n$ for $\pi$ fixed, is $\delta_0$ for $\pi < \delta_0$ and zero for $\pi > \delta_0$.
  From this it follows that, with probability $\delta_0$, the value $q_s$ measures the size of $|Q_s|$ in the random growth process.
  On the other hand, with probability $1 - \delta_0$, Algorithm~\ref{algo:erdos-renyi} begins in a component of order $O(\log n)$, and so the asymptotic density of the queue at time $t = \pi n$ and $\pi > 0$ is zero.
  Thus for $\pi > \delta_0$ the expected density of $|Q_{\pi n}|$ is zero, while for smaller $\pi$, the expected density is asymptotically $\delta_0 (1 - \pi - \exp(-d \pi))$.
  But of course, if we condition on choosing a vertex in the giant component then the expected density of $Q_t$ is asymptotically $(1 - \pi - \exp(-d \pi))$ if $\pi < \delta_0$.
\end{proof}

\begin{example}
  Theorem~\ref{randomgraph} is an asymptotic result.
  To see how well it holds in practice we conduct an experiment on a large random graph. We expect that, because of the quasirandomness of Ramanujan graphs, a random graph $G$ with the same number of vertices and edge density as $\LPS(13, 61)$ should exhibit the same behavior for the size of the queue as appears in Figure \ref{figure:RamanujanQueue}. To that end we generated an Erd\H{o}s--R\'{e}nyi random graph $G$ with 113,460 nodes and 794,220 edges from the uniform distribution of graphs with the given numbers of nodes and edges. The resulting graph had average degree 14.00 and the spectral gap of the normalized adjacency matrix was about 0.4848. We ran Algorithm~\ref{algo:growth} on $G$ and examined the behavior of the queue. Figure \ref{figure:randomExperiment} shows the results compared against the results for the Ramanujan graph already shown in Figure \ref{figure:RamanujanQueue}. We omit the curve given by Theorem \ref{randomgraph} with $p = 14/113460$, as it is visually indistinguishable from the curve modeling the behavior of the queue in the random graph experiment.

It is worth mentioning that our choice of edge density and number of vertices means that our actual randomly-generated graph ended up being connected. This is not surprising from the view point of the connectivity threshold of Erd\H{o}s--R\'{e}nyi random graphs. Indeed this classic result of \cite{ErdosReyni:1960} establishes that for $p > \frac{\log n}{n}$ asymptotically almost surely $G \sim G(n, p)$ will be connected. In our case $p = 14/113460$ and $14 > \log(113460) \approx 11.64$. To have a randomly-generated graph with edge density $14/n$ which is not connected we would expect to need more than 1.2 million vertices. Moreover by Theorem \ref{ERPhaseTransition}, such a graph would have more than 99.9999\% of its vertices in the giant component.
\end{example}

\begin{figure}[tb]
\centering
\includegraphics{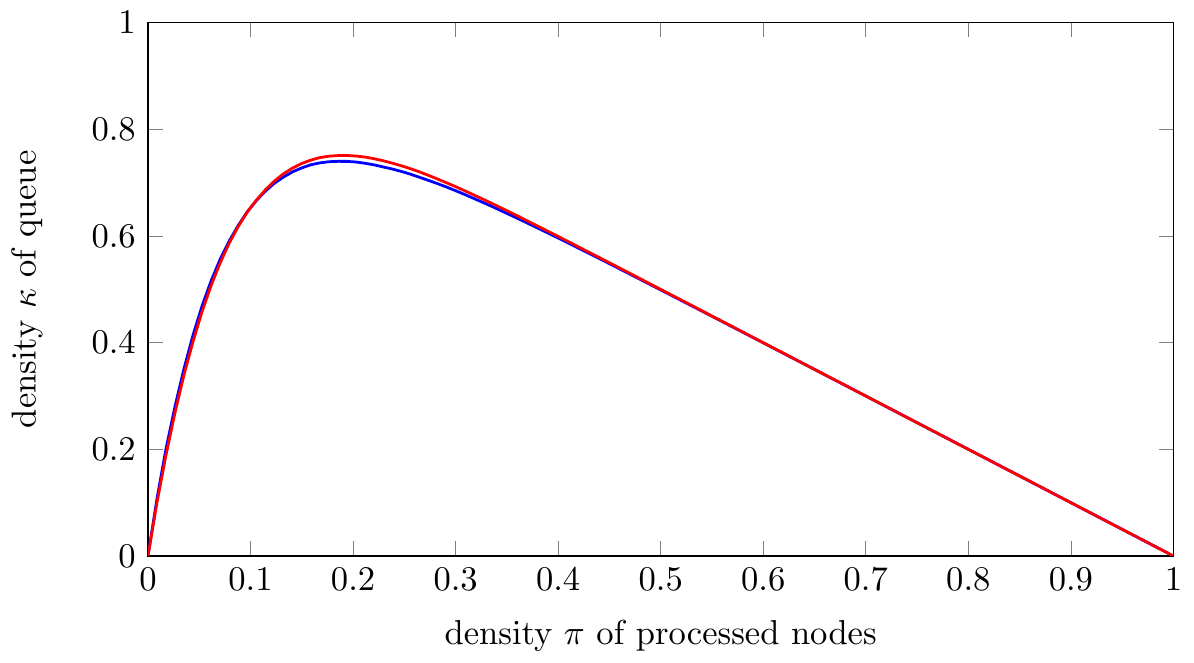}
\caption{The observed behavior of the queue in $G \sim G(113460, 14/133460)$ (blue) compared to the observed behavior of the queue in $\LPS(13, 61)$ (red). 
}
\label{figure:randomExperiment}
\end{figure}

\section{Enumerating triangulations in practice}\label{section:in_practice}

Standard software packages for enumerating triangulations include \topcom \cite{topcom}, \gfan \cite{gfan} and \mptopcom \cite{mptopcom}.
Both \topcom and \gfan employ random growth through the (regular) flip graph, whereas \mptopcom uses reverse search \cite{AF93}.
We would like to estimate when these programs will terminate early in the enumeration process.
To this end we try our methods developed for regular and nonregular graphs with good expansion.
In contrast to the previous sections our observations here are based on experiments only.

Many interesting point configurations are highly symmetric.
If there is a group acting on $\pointconf$, its action naturally lifts to the set of triangulations of~$\pointconf$.
Hence, one considers the \emph{quotient flip graph}, having orbits of triangulations as nodes.
Two orbits are connected by an edge if there is a pair of representatives from both orbits, which differ by a flip.
It suffices to visit the nodes of the quotient flip graph; orbits can be expanded later if needed (e.g., in an embarrassingly parallel manner).
All software systems mentioned support computing in quotient flip graphs.

For instance, a regular $k$-gon admits the natural action of the dihedral group $D_k$ by rotations and reflections.
Figure~\ref{fig:ngon} shows the quotient flip graph of a hexagon by the cyclic group of order six, which is a (normal) subgroup in $D_6$ of index two.
It has four nodes, and the average degree is $\overline{d}=(1+2\cdot 2+3)/4 = 2$.

There is no reason to believe that a general flip graph is a good expander.
Moreover, most flip graphs are not regular.
The same holds true for any of its quotients.
Yet, we will argue heuristically that the situation is maybe somewhat more benign, from a practical point of view.
Since flip graphs tend to be very large, in practice we will mostly encounter point configurations which are rather small, both in terms of dimension and number of points.
In many cases this will mean that the point configuration in question admits a substantial number of splits, even if there are many other coarsest subdivisions.
We believe that this should result in a rather low variance in the degree sequence of flip graphs within computational reach.
Moreover, the vast majority of the triangulations of a point configurations exhibits no symmetry at all, which means that most orbits of triangulations have the size of the entire group acting.
We believe that this should force that also the quotient flip graphs within computational reach still have a moderate variance in their degree sequence.

\begin{figure}[htb]
\centering
\includegraphics{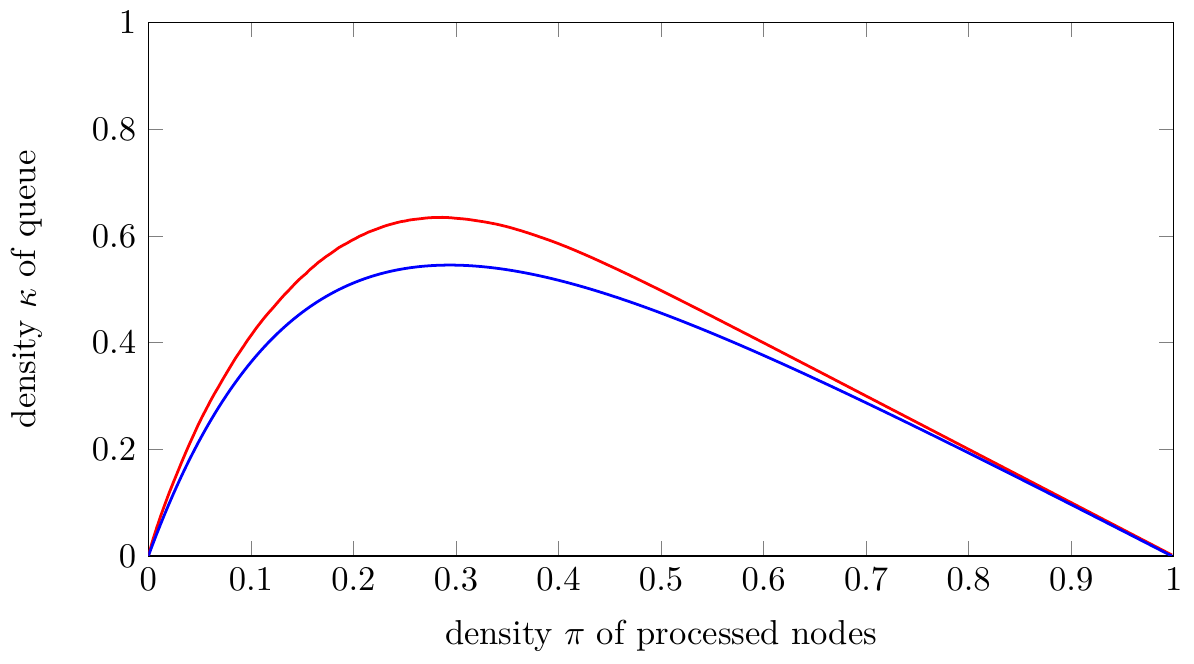}
\caption{
Lower bound on expected queue size from Theorem \ref{thm:main2} (blue) and a comparison to experimental data (red) for the flip graph of the 4-cube.
}
\label{figure:4cubeComparison}
\end{figure}

\begin{example}
  To see how strong of a comparison we can make between a flip graph and a non-regular Ramanujan graph, as in the discussion in Section \ref{sec:nonregular}, we consider the quotient flip graph of \emph{all} triangulations of the 4-cube, which is connected \cite{Pournin:2013}.
  The $4$-dimensional cube, with 16 vertices, has $92{,}487{,}256$ triangulations, partitioned into $247{,}451$ orbits with respect to the natural action of the Coxeter group of type B$_4$, whose order is $384$.
  Exactly $234{,}703$ orbits have the maximal length $384$; this is more than $94.8\%$.
  The corresponding numbers counting only regular triangulations can be found in \cite[\S6.3.5]{Triangulations}.

  The flip graph of all triangulation of the 4-cube has the $247{,}451$ orbits of triangulations as its vertices and it has $1{,}548{,}472$ edges. Thus it has average degree $\overline{d} \approx 12.5154$. In Figure \ref{figure:4cubeComparison} we show a comparison to the size of the queue as a function of the density of processed vertices, as the random growth process runs on this flip graph, compared to the lower bound given by $\beta(\upsilon, 12.5154,2\sqrt{12.5154 - 1})$ as $\upsilon$, the density of processed nodes, varies from 0 to 1.
  Visually, our heuristic works well in this case as the two curves are quite similar.
  
  Like the flip graphs of $k$-gons for higher $k$ in the regular setting, the quotient flip graph of the 4-cube itself does not actually have good expansion. The second largest eignvalue of the normalized adjacenty matrix is .9846 from computation in \matlab \cite{MATLAB:2019a}. Nonetheless, however, the two curves in Figure \ref{figure:4cubeComparison} match up well. 
  
\end{example}

We don't have a rigorous explanation for why random growth on flip graphs in our examples is so closely modeled by the behavior in regular Ramanujan-graphs in Theorems \ref{thm:main2} and \ref{thm:VertexCounts}. One possible reason is that the assumption of large spectral gap is stronger than what is needed to reach our conclusions. It is known that a $d$-regular Ramanujan graph on $n$ vertices has diameter roughly $\log n$ when $n$ is large (this is discussed, e.g., in \cite{HLW}).
While Table~\ref{tab:flip-k-gon} seems to suggest that flip graphs of $k$-gons do not have large spectral gap, in general, Sleator, Tarjan, Thurston \cite{SleatorTarjanThurston:1988}, and Pournin \cite{Pournin:2014} show they do have low diameter (again logarithmic in the number of vertices).
Perhaps low diameter is enough to imply that random growth works similarly to how it works in a similarly-sized Ramanujan graph.

We close this paper by returning to the triangulations of the 15-gon.
\mptopcom is based on the reverse search method of Avis and Fukuda \cite{AF93}, which is memory efficient and easy to parallelize \cite{mts}.
That algorithm implicitly picks a spanning tree of the (quotient) flip graph, in a deterministic way.
We applied Hall--Knuth random sampling \cite{HallKnuth:1965,Knuth:1975} to \mptopcom's reverse search tree of the flip graph of the 15-gon to estimate the size of the graph.
Figure~\ref{fig:knuth} shows how increasing the number of samples makes that estimate more precise.
Comparing with Figure~\ref{figure:15gon} is a bit difficult, since the $x$-axes are not the same.
Yet it seems that our estimates based on random growth and spectral expansion are superior, at least in this case.

\begin{figure}[htb]
\centering
\includegraphics{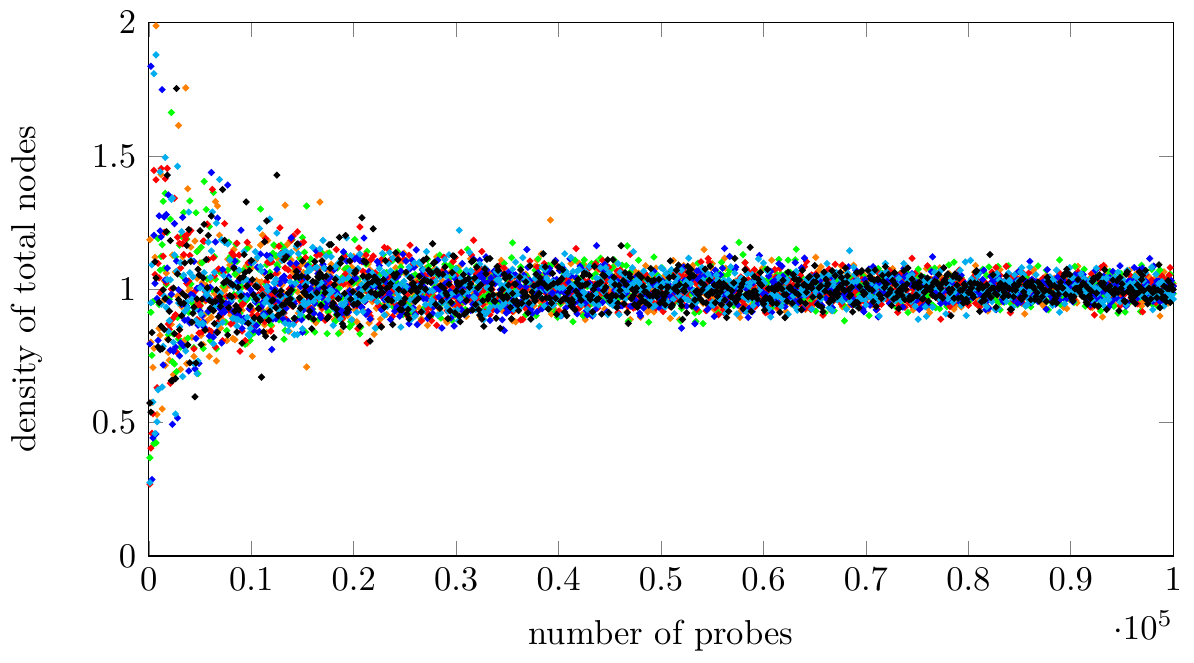}
\caption{
  Hall--Knuth estimates for the reverse search tree in the flip graph of the 15-gon.
  Every point stands for one estimate.
  The $y$-coordinate of the point is the estimate divided by the actual number of nodes ($742,900$).
  The $x$-coordinate indicates how many random probes into the tree were made.
}
\label{fig:knuth}
\end{figure}


\section*{Acknowledgments}
The authors thank Marek Kaluba for generating the examples of Lubotzky--Phillips--Sarnak Ramanujan graphs used in Section \ref{sec:RegularGraphs}, see \cite{LPS}.
Additionally, the authors thank Stefan Felsner, Christian Haase, Nati Linial, and Tibor Szab\'{o} for helpful discussions and comments on an earlier draft of this paper.

\printbibliography
\end{document}

%% file: ngon_lambdas.tex
\begin{table}[bh]
\caption{Data on the flip-graphs of $k$-gons, for $k\geq 15$.
  Here $n$ is the number of nodes, i.e., triangulations, $m$ is the number of edges, i.e., flips, and $d$ is the degree.
  The small cases $k\in\{3,4\}$ are omitted since those flip graphs are bipartite (in fact, consisting of a single node for $k=3$, and a single edge for $k=4$).
}
\label{tab:flip-k-gon}
\begin{tabular}{crrrrrrrrrr}
\toprule
$k$         & 5        & 6       & 7      & 8      & 9       & 10        \\
\midrule
$n$            & 5        & 14      & 42     & 132    & 429     & 1{,}430   \\
$m$            & 5        & 21      & 84     & 330    & 1{,}287 & 5{,}005   \\
$d$                & 2        & 3       & 4      & 5      & 6       & 7         \\
$\lambda$          & 1.6180   & 2.4142  & 3.2320 & 4.3834 & 5.4885  & 6.5650    \\
$2\cdot\sqrt{d-1}$ & 2.0000   & 2.8284  & 3.4640 & 4.0000 & 4.4720  & 4.8988    \\
\toprule
$k$         & 11      & 12       & 13        & 14            & 15 \\
\midrule                     
$n$            & 4{,}862 & 16{,}796 & 58{,}786  & 208{,}012     & 742{,}900 \\
$m$            & 19{,}448& 75{,}582 & 293{,}930 & 1{,}144{,}066 & 4{,}457{,}400 \\
$d$                & 8       &  9       & 10        & 11            & 12\\
$\lambda$          & 7.6228  &  8.6678  & 9.7038    & 10.7331       & 11.7574\\
$2\cdot\sqrt{d-1}$ & 5.2914  &  5.6568  & 6.0000    & 6.3244        & 6.6332\\
\bottomrule
\end{tabular}
\end{table}